\documentclass{amsart}

\usepackage[parfill]{parskip}  
\usepackage{amssymb}
\usepackage{mathtools}
\usepackage{comment}
\newtheorem{sen}{Proposition}
\newtheorem {M-lemma}{Lemma} 
\newtheorem* {m-theorem}{Main theorem}

\begin{document}
\title{On a lattice of relational spaces (reducts) for the order of integers}
\author[A.L. Semenov, S.F. Soprunov]{A.L. Semenov, S.F. Soprunov}
\begin{abstract} 
We investigate the definability (reducts) lattice of the order of integers and describe a sublattice generated by relations 'between', 'cycle', 'separation', 'neighbor', '1-codirection', 'order' and equality'. Some open questions are proposed.
\end{abstract}
\maketitle

\section{Introduction}
In his works of the beginning of the last century \cite{hun} Edward Huntington pointed out three relations that can be defined through the relation of order on any ordered set. These relations are:
\begin{align*}
B(a,b,c)&\rightleftharpoons (a<b<c)\lor(a>b>c)\\
C(a,b,c)&\rightleftharpoons (a<b<c)\lor(b<c<a)\lor(c<a<b)\\
S(a,b,c,d)&\rightleftharpoons \Big(B(a,b,c)\lor B(a,d,c)\Big)\land\Big(B(b,a,d)\lor B(b,c,d)\Big)
\end{align*}

We fix the universe on which the relation of order and all relations defined through it are set.
As it is well known (see, for example, the work of the authors of \cite{we}), classes of relations closed by definability (definability spaces) form a lattice. If the original order relation and equality relation are added to the spaces corresponding to the Huntington relations, then the following lattice is obtained:

\begin{figure}[h]
     \centering
     \includegraphics[width=70mm]{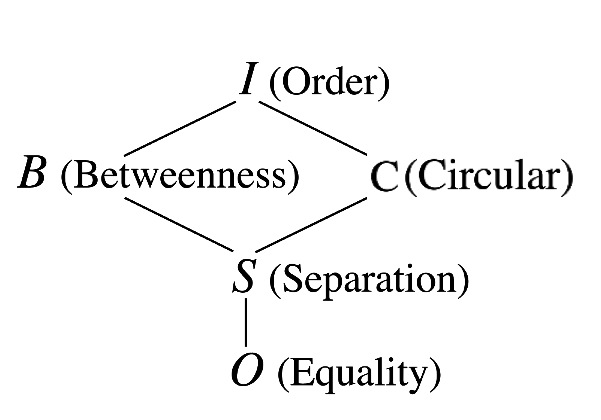}
     \centering
     \caption{Lattice for $\langle \mathbb Q, < \rangle$}
     \label{fig:Q-fig}
 \end{figure}

Of course, for some linearly ordered sets, elements of this lattice may coincide. However, for example, for rational numbers it is not difficult to prove that all elements are different.

More relations can be defined through the order. In particular, the relation of "codirectional" or "equipollence" of segments has a clear geometric meaning:

\begin{align*}
E(a,b,c,d)&\rightleftharpoons \Big((a<b)\land(c<d)\Big)\lor\Big((a>b)\land(c>d)\Big)
\end{align*}

However, it is not difficult to see that it generates the same definability space as "between". 

Through the order (and even through the "cycle"), it is possible to determine the "successor" - the smallest element larger than the given one, and the correspondent binary relation "succession". However, in some structures, for example, in the order of rational numbers, this relation degenerates. The situation is similar for the symmetrical "neighborhood" relation, with the meaning "nothing between two elements".

Finally, in the case of "discrete" orders, when the relations "succession" and "neighboring" are nontrivial, a series of relations "successor to the successor" arise, etc. For the case of integers with successor, we described the definability lattice  in \cite{we}. Namely, the elements of this lattice are divided into three series parameterized by natural number $n$: <<$+n$>>, <<distance $n$>>, <<$n$-codirection>> (4-ary relation) (exact definitions are in \cite{we}). 

Later on, we use the "neighborhood" relationship, which is the same as "distance 1". We will also use "1-codirection", which can  be defined as the co-directionality of neighbors. 

After the integers with successor the natural next step is to consider the order on integers.

In this paper, we study the definability sublattice of the order of integers located between the "1-codirection" and the order. 

Of course, the definability lattice for the order of integers contains the whole lattice of the integers with successor. 
 \section{Preliminary remarks}

 In this paper we continue to use a special case of Svenonius theorem for structures that have upward complete (elementary) extension.

 We use the ‘topology of pointwise convergence’ of the permutation group on an infinite countable set $S$. So, we naturally define closed (in this topology) subgroups of $Sym(S)$ - the group of all permutations of $S$.

 It is known (see, for example, \cite{max})

\begin{sen}
  The lattice of definability spaces of a structure having a complete upward extension is anti-isomorphic to the lattice of automorphism groups   of the spaces.
\end{sen}

We will sometimes  call the space generated by the relation $R$ simply $R$, or the space $R$.

In the following text, "definability" means definability in the structure of integers with order. 

 The structure $\langle\mathbb Z, <\rangle$ 
 has an upward complete extension, see \cite{max}. 
 It is $\langle \mathbb Q, < \rangle\times\langle\mathbb Z,+1\rangle$. 
 We will denote it $\mathbb{Q\times\mathbb Z}$. 
We can imagine copies of integers as vertical lines placed on the horizontal axis of rationals.

Thus, the lattice of definability spaces of the structure $\langle\mathbb Z, <\rangle$ is completely determined by the lattice of closed supergroups of the shifts'  (monotonically increasing permutations) group of the
 structure $\mathbb{Q\times\mathbb Z}$.

 Recall that the lattice of spaces of the structure $\langle\mathbb Q ,<\rangle$, as well as the description of the corresponding permutation groups $\mathbb Q$ have long been known e.g.
 \cite{cam}).



  

 
 The lattice of the structure $\langle\mathbb Z, +1\rangle$ is fully described by the authors of this article in \cite{we}.

 \section{ The main theorem. Formulation}
 
 Following [4], we introduce the notation
 
${A_1}_n(x,y,z,t)\rightleftharpoons (x-y)=(z-t)\land |x-y|=n$.

Of course, we mean the corresponding formula in the signature of order. 
In our structure, if $x$ and $y$ lie on different verticals, the relation ${A_1}_n$ is false, the same for $z$ and $t$.

As we said above, this relation is naturally called $n$-\emph{codirection}.

Thus, ${A_1}_1$ is 1-co-direction.
 
In this paper, we prove that the lattice of spaces of the structure $\langle\mathbb Z, <\rangle$ contains the following sublattice:

 \begin{figure}[htp]
    \centering
\includegraphics[width=50mm]{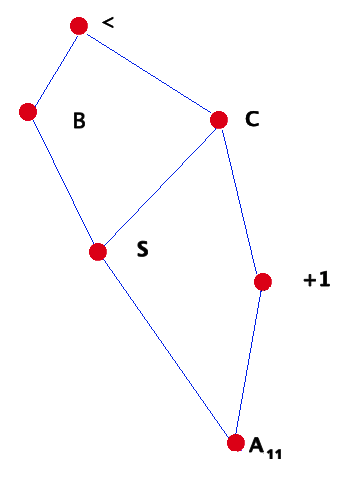}
   \caption{The Diagram}
    \label{z-fig}
\end{figure}


The main result of the paper is:
\begin{m-theorem}\label{m-sen}
  $ $  

The Diagram (Fig.\ref{z-fig}) shows all definability spaces, greater than ${A_1}_1$ and smaller than <. All relations represented by the edges of the Diagram are strict inclusions (of the space located below into the space located above). All other inclusion relations between the vertices of the Diagram are contained in its transitive closure.

\end{m-theorem}


It is possible that there are elements in the lattice that are not comparable to the 3 non-trivial Huntington relations and are smaller than co-direction, but non of them is known to us.

The fact of inclusion of spaces connected by a segment of the Diagram can be
verified directly by constructing a formula. \\

The main idea of our proof is that properties of the sublattice in consideration are in some sense determined by the properties of the lattices of the structures $\langle\mathbb Q ,<\rangle$ and $\langle\mathbb Z, +1\rangle$.

\section{The main theorem. Proof}\label{disc}

We begin by describing the automorphism groups of the relations shown in Fig. \ref{z-fig}

If $R$ is a definable relation, then by $\Gamma(R)$ we denote the group of automorphisms (the group of permutations) of the structure $\mathbb Q\times\mathbb Z$ preserving this relation. In cases where we are talking about a relation on another structure, the structure is explicitly indicated as an argument: so $\Gamma(C)$ is the automorphism group of the "cycle" relation on $\mathbb Q\times\mathbb Z$, and $\Gamma(\mathbb Q,C)$ is the group automorphisms of the "cycle" relation on the structure $\langle\mathbb Q, <\rangle$.

Let's call a permutation $g$ acting on $\mathbb Q\times\mathbb Z$, \emph{systemic} if the image of each vertical $a$ coincides with some vertical $b$. We say that systemic  $g$ \emph{initiates} on $\mathbb Q$ the permutation that maps $a$ to $b$. Similarly, the term \emph{initiates} is used for  groups of systemic permutations.

We call a systemic permutation \emph{positive} if it preserves the order on each vertical and \emph{negative} if it reverses the order on each vertical.

Let us call a positive permutation \emph{shift} if it initiates an increasing mapping on $\mathbb Q$. The \emph{shifts'} group is the group of all shifts.  We call this permutation \emph{vertical shift} if it initiates identical mapping on $\mathbb Q$.

\begin{M-lemma}\label{sys}
$ $
\begin{enumerate}
\item The initiation mapping acts as a homomorphism of the lattice of supergroups of systemic permutations to the lattice of closed supergroups of $\Gamma(\mathbb Q,<)$. On each systemic supergroup initiation acts as a group homomorphism.

\item The group of automorphisms of the "neighborhood" coincides with the group of systemic permutations.

\item The group of "1-codirection" automorphisms coincides with the group of all positive or negative permutations.

\item The group $\Gamma(+1)$ coincides with the group of all positive permutations.

\item The group $\Gamma(<)$ coincides with the shifts' group.
\end{enumerate}
\end{M-lemma}

\begin{proof}
The first statement obviously follows from the definition. For the rest of the statements, the "right-to-left" proof consists in checking the conservation of relations in permutations. For example, it is clear that the neighborhood relation holds only for two elements of the same vertical and is preserved if the order on the vertical is preserved or reversed. 

The "left-to-right" proof includes, for example, the following consideration. Consider a mapping that preserves the neighborhood, then it translates two neighboring elements of the same vertical into two neighboring ones of some vertical, preserving or reversing their order. Induction and considerations of mutual unambiguity show that the vertical moves to the vertical and the order will be preserved or reversed on all elements of the vertical. 
\end{proof}

Note that the group from the previous lemma p.4 (as well as pp. 2, 3) initiate the  group $Sym(\mathbb Q)$ of all permutations of $\mathbb Q$. This is not surprising -- they correspond to "local" relations. 

So, in this paper we study the closed groups lying between the group of shifts (p.5) and the group of positive and negative permutations (p.3).

As we noted above, the structure of the definability lattice for the order of rational numbers is known (see for example \cite{cam}).

\begin{sen}\label{q-sen}
The group $Sym (\mathbb Q)$ contains exactly 5 closed subgroups containing the groups of shifts of $\mathbb Q$:

$$Sym(\mathbb Q),\Gamma(\langle \mathbb Q, S\rangle), \Gamma(\langle \mathbb Q, B\rangle), \Gamma(\langle \mathbb Q, C\rangle), \Gamma(\langle \mathbb Q, <\rangle)$$

 (see Fig. \ref{fig:Q-fig})
 \end{sen}




Further in this section, $\Gamma^*$ is a closed subgroup of the group $\Gamma({A_1}_1)$. In particular, all elements of $\Gamma^*$ are systemic elements and initiation is defined for them. We also remember that $\Gamma^*$ is a supergroup of the shifts' group.

Multiplication of element of $\Gamma({A_1}_1)$ by arbitrary vertical shift does not affect the initiated element of $Sym (\mathbb Q)$. Up to this multiplication every element of $Sym (\mathbb Q)$ has exactly two elements of $\Gamma({A_1}_1)$ that initiate it.

According to the lemma \ref{sys} and the statement \ref{q-sen}, the group $\Gamma^*$ initiates one of the 5 specified groups.

We will consider each of these options and describe the group $\Gamma^*$ for them. It is essential, of course, that  different  groups of automorphisms can initiate the same group.

It is easy to notice that each permutation from the group $Sym(\mathbb Q)$ is initiated up to vertical shifts exactly by one positive and one negative element of the group $\Gamma({A_1}_1)$. Therefore, we have 
\begin{M-lemma}\label{triv}
 Each group of $\Gamma^* \subset \Gamma({A_1}_1)$ is completely determined by the initiated subgroup of the group $Sym(\mathbb Q)$ and indication which elements of this subgroup are initiated by a positive and which by a negative permutation from $\Gamma^*$.  
\end{M-lemma}

Let us note, that one element can be initiated to two different elements. This happens for elements of the group $\Gamma({A_1}_1)$.
We will use the following trivial
\begin{sen}\label{gr-sen} 
If the group $\Gamma^*\supset\Gamma(+1)$ contains at least one negative permutation, then it contains all negative permutations.
\end{sen}

\begin{proof}
 Let $g$ be a given negative permutation, and $h$ be an arbitrary negative one. Then $g^{-1}$ is negative, $g^{-1}\circ h$ is the positive element of $\Gamma^*$. So $g \circ g^{-1} \circ h $ lies in $\Gamma^*$.
\end{proof}

To analyze all 5 options, we will use lemma \ref{ss}  below, and the concept of $m$-indistinguishability. The lemma \ref{m} establishes, using this concept, in a sense, the "locality" of all relations definable in our basic structure. Similar constructions were used in \cite{we}.

Let's add to $\mathbb Z$ element $\infty$  with the  obvious semantics, we assume that $\infty>m$ for any natural $m$ and $|\infty|=\infty$.

Let's define the function "$-$" that maps $(\mathbb Q\times\mathbb Z)^2$ to $\mathbb Z \cup \{{\infty}$\} so that $a-b=\infty$ if $a,b$ lie on different verticals and $\langle r,z_1\rangle-\langle r,z_2\rangle=z_1-z_2$.

If $m$ is a natural number, then two vectors $a_1,a_2,\dots,a_n$ and $b_1,b_2,\dots,b_n$ of equal length are called \emph{m-indistinguishable} if two requirements hold for all $i,j\leqslant n$:
$$a_i<a_j\equiv b_i<b_j,$$
 $$(|a_{i}-a_j|\leqslant m\lor|b_{i}-b_j|\leqslant m)\to (a_{i}-a_j=b_{i}-b_j).$$
 

Let's note a few simple properties of $m$-indistinguishability:

(i) if the vectors $\bar a,\bar b$ are  $m$-indistinguishable, then they are $k$-indistinguishable for every $k<m$.

(ii) if the vectors $\bar a=\langle a_1,\dots, a_n\rangle,  \bar b=\langle b_1,\dots,b_n\rangle$ are $m$-indistinguishable for all natural $m$, then the mapping $\varphi(a_i)=b_i$ can be extended to a
 shift, therefore $R(\bar a)\equiv R(\bar b)$ for every definable $R$.

Let's call the number $m$ \emph{the boundary} for a relation $R$ if $R(\bar a)\equiv R(\bar b)$ for every $m$-indistinguishable $\bar a$ and $\bar b$.

We shall need an analogue of Lemma 4.4 from \cite{we}. 

More precisely, 
\begin{M-lemma}\label{m}
 Every definable relationship has a boundary.
\end{M-lemma}

\begin{proof}
    The proof is similar to the proof of Lemma 4.4 from \cite{we}, we present it here for completeness.

    Consider the enrichment $Z^+$ of the  structure of $\langle Z,<\rangle$ by the ternary relation +. Of course, all integers are definable in the resulting structure. For every natural $n$ the relation $E_n(m,\bar a,\bar b)$ , stating that the vectors $\bar a,\bar b$ of length $n$ are $m$-indistinguishable, is definable. In the structure $\langle Z,<\rangle$ for each fixed $n,m$, the binary relation $E^*_{n,m}(a,b)$ , equivalent to $E_n(m,\bar a,\bar b)$ is definable; for each $m$ by  a different formula. 
    
    Let's take some non-standard elementary extension $Z^*$ of the  structure $Z^+$. Let $R$ be some $n$-ary relation definable in $\langle Z, <\rangle$. We will show that 
   
    $$ (*) Z^*\vDash (\forall \bar a,\bar b)(E_n(m,\bar a,\bar b)\to (R(\bar a)\equiv R(\bar b))$$ for some natural $m$.

    Indeed, due to properties (i), (ii) in the definition of $m$-indistinguishability, the property $(*)$ is fulfilled for a positive non-standard element $m_0$, and, consequently, for some natural $m$. Then 
    $$\langle Z,<\rangle\vDash (\forall \bar a,\bar b)(E^*_{n,m}(\bar a,\bar b)\to (R(\bar a)\equiv R(\bar b))$$
    which is equivalent to the statement of the lemma.    
\end{proof}
\begin{M-lemma}\label{ss}

Let a group $\Gamma^*$ be
a supergroup of the shift group and 

(i) if $\Gamma^*$ contains such a positive permutation $g$, that for some four verticals $\alpha<\beta<\gamma<\delta$ holds $g(\alpha)>g(\beta)>g(\gamma)>g(\delta)$,

then $\Gamma^*\supset \Gamma(+1)$.

(ii)  $\Gamma^*$ contains such a negative permutation $g$, that for some four verticals $\alpha<\beta<\gamma<\delta$ holds $g(\alpha)<g(\beta)<g(\gamma)<g(\delta)$,

then $\Gamma^*\supset \Gamma({A_1}_1)$.

In particular, $\Gamma^*$ initiates $Sym(\mathbb Q)$.

\end{M-lemma}
\begin{proof}
Let $R$ be an arbitrary definable relation preserved by the group $\Gamma^*$. Let the vector $\bar a=\langle a_1,a_2,\dots ,a_p\rangle$ has the form
$\langle \bar   b_1, \bar b_2,\dots\bar b_n\rangle$, where for each $i$ all elements of $\bar b_i$ lie on the same vertical, all these verticals are different, and 
$\bar b_1<\bar b_2<\dots \bar b_n$ (here inequality, of course, holds between elements of verticals).
Let's call \emph{blocks} all these $\bar b_i$ .

We are going to show that if the definable relation $R$ is true on the vector $\langle\bar b_1, \bar b_2,\dots\bar b_n\rangle$, then it is true on each vector obtained by permutation of blocks $\bar b_i$. 
To do this, we will use "partial isomorphism" -- embed the elements of the vector $\bar a$ into the main structure $\mathbb Q\times\mathbb Z$, preserving the order between these elements, and preserving the distances inside one block, we will say: "place" the vector $\bar a$.
$ $\\

Case (i)
\begin{proof}
Let the mapping $g$ satisfy the condition (i) of the lemma, i.e. $g$ changes the order of the verticals $\alpha<\beta<\gamma<\delta$, preserving the order on each of them.
    Let's take arbitrary $k,l,m$ such that  $1\leqslant k \leqslant l\leqslant m\leqslant n$ and divide the elements  $\bar b_1<\bar b_2<\dots \bar b_n$ into four groups 

$\{\bar b_i|1\leqslant i \leqslant k\};\{\bar b_i|k < i \leqslant l\};\{\bar b_i|l < i \leqslant m\};\{\bar b_i|m< i \leqslant n\}$.

Let's "place" the blocks from the first group on the vertical $\alpha$, the second on the vertical $\beta$, the third on the vertical $\gamma$, the fourth on the vertical
$\delta$ so that the distance between adjacent blocks on the same vertical exceeds the boundary for the ratio $R$.
We get a vector $\bar a'$ such that $R(\bar a) \equiv R(\bar a')$

 Let 
$\bar b'=g(\bar a')$, so $R(\bar b')\equiv R(\bar a')$. Then:
$ $\\

(*) in the vector $\bar b'$, the blocks go in the  order:\\ $\bar b_{m+1},\dots,\bar b_n,\bar b_{l+1},\dots, \bar b_m,\bar b_l,\dots, \bar b_{k+1},\bar b_1,\dots,\bar b_k$.\label{prf}
$ $\\

It is clear that by applying such vector transformations for different $k,l,m$ we can obtain every predetermined order of blocks.

Thus, any relation preserved by the group $\Gamma^*$ is preserved by any finite permutation of $\Gamma(+1)$.

The group $\Gamma^*$\underline{is closed}, so any permutation preserving the family of relations defined by +1 belongs to the group. Thus $\Gamma^*\supset\Gamma(+1)$.

Case (i) has been considered.
\end{proof}
$ $\\

Case (ii)
\begin{proof}
Let $g$ preserve the order of the verticals $\alpha<\beta<\gamma<\delta$, reflecting the order on each of them. Let's "place" the blocks $\bar b_1,\bar b_2,\dots,\bar b_n$ on the vertical $\alpha$ so that the distance between adjacent blocks exceeds the boundary for the relation $R$ and apply the permutation $g$. Since this permutation reverses the order on $\alpha$, we get the vector $\bar a'=\bar b'_n<\bar b'_{n-1}<\dots<\bar b'_1$ and each block $\bar b'_i$ is also "inverted" -- elements in the block $\bar b'_i$ go in reverse order compared to $\bar b_i$.

Let's take arbitrary $k,l,m$ such that $1\leqslant k\leqslant l\leqslant m\leqslant n$ and divide the elements $\bar b'_n<\bar b'_{n-1}<\dots<\bar b'_1$ into four groups 

$\{\bar b'_i|l+1\leqslant i \leqslant n\};\{\bar b'_i|k+1\leqslant i \leqslant l\};\{\bar b'_i|1\leqslant i \leqslant k\}$.

Let's take the verticals $\alpha,\beta, \gamma,\delta$ and "place" the blocks from the first group on the vertical $\alpha$, the second on the vertical $\beta$, the third on the vertical $\gamma$, the fourth on the vertical
$\delta$ so that the distance between adjacent blocks on the same vertical exceeds the boundary for the ratio $R$, we get the vector $\bar a'$, apply the permutation $g$ and get the vector $\bar b=g(\bar a')$. Since this permutation reverses the order on each of these verticals, the original order of the blocks in each group will be restored in the vector $\bar b$, as well as the original order of the elements in each block. In other words, the vector $\bar b$ has the form $\bar b_{m+1},\dots, \bar b_n,\bar b_{l+1},\dots,\bar b_m, \bar b_{k+1},\dots,\bar b_l,\bar b_1,\dots,\bar b_k$.  Thus, the configuration (*) from part (i) of the proof is obtained.

Similarly to the case (i) and using Proposition \ref{gr-sen} we obtain $\Gamma^*\supset \Gamma({A_1}_1)$.

The proof in case (ii) has been completed.
\end{proof}

In both cases $\Gamma^*$ initiates $Sym(\mathbb Q)$.
The proof of Lemma \ref{ss}  is completed.
\end{proof}

$ $\\

Now, we come to the consideration of the listed variants of subgroups initiated by the group $\Gamma^*$:

$ Sym(\mathbb Q)$ (Equality),

$\Gamma(\langle \mathbb Q, S\rangle)$ (Separation),

$\Gamma(\langle \mathbb Q, B\rangle)$ (Between), 

$\Gamma(\langle \mathbb Q, C\rangle)$ (Cycle),

$\Gamma(\langle \mathbb Q, <\rangle)$ (Order).
$ $\\

For each of these variants we know at list one group that initiates it. Our problem is to find all such (closed) groups.
$ $\\

\textbf{$\Gamma^*$ initiates $Sym(\mathbb Q)$ -- Equality}

Recall that the group in comsideration is a subgroup in $\Gamma({A_1}_1)$. Therefore, any permutation in it is either positive or negative.
If all permutations of $\Gamma^*$ are positive, then according to Lemma \ref{sys}, clause 4, this is the group $\Gamma(+1)$. 

Let there be at least one negative among the permutations.

Since $\Gamma^*$ initiates $Sym(\mathbb Q)$, then any permutation from $Sym(\mathbb Q)$ is initiated by some (positive or negative) permutation from $\Gamma^*$. Let's take an arbitrary permutation $g$ of $Sym(\mathbb Q)$, which simultaneously preserves the order of some 4 elements of $\mathbb Q$ and reverses the order of some 4 (other) elements.  According to the lemma \ref{ss}, whichever (positive or negative) permutation of $\Gamma^*$ initiates $g$,
 $\Gamma^*\supset\Gamma(+1)$ holds.

So, according to the statement \ref{triv}, $\Gamma^*$ contains all negative permutations, inducing permutations from Lemma \ref{ss}(ii). 
This implies, according to the lemma \ref{ss}, that $\Gamma^*$ coincides with $\Gamma({A_1}_1)$.

There are no other options. Thus, a variant of $Sym(\mathbb Q)$ has been considered and has given to us two spaces of the Diagram 2.
$ $\\

So, we continue with other four cases.
$ $\\

We will use a definition from \cite{perm} for the special case of the order of rational numbers.

\emph{A section} is a pair of $\{I_1,I_2\}$ subsets of $\mathbb Q$ (one of which may be empty) such that $I_1 \cup I_2=\mathbb Q;$ and for all   $a\in I_1, b\in I_2$ holds $a<b$.
$ $\\

The following description of the elements of the group $\Gamma(\langle\mathbb Q, S\rangle)$ is given in the main theorem of \cite{perm}:

\begin{sen}

The permutation $h\in Sym(\mathbb Q)$ belongs to $\Gamma(\langle\mathbb Q, S\rangle)$ if and only if exactly one of two conditions is met for a certain section $\{I_1,I_2\}$:
$ $\\

(i) $g$ decreases by both $I_1$ and $I_2$, preserving their order\\

(ii) $g$ increases by both $I_1$ and $I_2$, reversing their order.
\end{sen}

\textbf{$\Gamma^*$ initiates $\Gamma(\langle\mathbb Q, S\rangle)$ -- Separation}

Let's show that
\textbf{the group $\Gamma(S)$ consists of all negative permutations initiating permutations satisfying (i) and all positive permutations initiating permutations satisfying (ii)}.

\begin{proof}
$ $

It is easy to see that
the group satisfying the specified description preserves the relation $S$. 
$ $

On the other hand, such a group is maximal among those preserving the relation $S$. To prove this, consider its arbitrary proper supergroup $G$. 
$ $\\

If $G$ initiates a proper supergroup of the group $\Gamma(\langle\mathbb Q,S\rangle)$, then this  supergroup coincides with $Sym(\mathbb Q)$. Indeed,
 as is known (Fig.\ref{fig:Q-fig}), there is no intermediate group between the groups $\Gamma(\langle \mathbb Q,O\rangle)=Sym(\mathbb Q)$ and $\Gamma(\langle\mathbb Q,S\rangle)$. This option was discussed above -- \textbf{$\Gamma^*$ initiates $Sym(\mathbb Q)$}.

If the supergroup $G$  initiates $\Gamma(\langle\mathbb Q,S\rangle)$, then it contains a positive and a negative permutations, say $g_+$ and $g_-$ that initiate the same element $h$ of the group $\Gamma(\langle\mathbb Q, S\rangle)$.

If for $h$ condition $(i)$ of Proposition 4 holds then $g_+$ decreases on some infinite element of the section and for $g_+$ holds $(i)$ of lemma \ref{ss}.

If for $h$ condition $(ii)$ of Proposition 4 holds then $g_-$ increases on some infinite element of the section and for $g_-$ holds $(ii)$ of lemma \ref{ss}.

In the both cases the proper supergroup initiate $Sym(\mathbb Q)$).
\end{proof}

\textbf{$\Gamma^*$ initiates $\Gamma(\langle\mathbb Q, B\rangle)$ -- Between}

We show that \textbf{$\Gamma(B)$ consists of all positive permutations initiating the shift (increasing mapping on $\mathbb Q$) and all negative ones initiating the reflection (decreasing mapping on $\mathbb Q$).}
\begin{proof}

As in the Separation variant, we directly check that such a group preserves the relation $B$.

Further, as in the Separation variant, we note that this group is maximal among such groups. Indeed, let $H$ a group containing it. $H$ can initiate one of the supergroups $\Gamma(\langle\mathbb Q, B\rangle)$. That was considered above. $H$ can have an element satisfying one of the conditions (i) or (ii) of the lemma \ref{ss}, so the supergroup initiates $Sym(\mathbb Q)$. This we considered also.
\end{proof}

\textbf{$\Gamma^*$ initiates $\Gamma(\langle\mathbb Q, C\rangle)$ -- Cycle}

We show that \textbf{the group $\Gamma(C)$ consists of all positive permutations  from $\Gamma(S)$ 
}.

\begin{proof}

We directly check that the described group group preserves the relation $C$.

We note that this group is maximal among such groups -- the groups containing it must either initiate the supergroups $\Gamma(\langle\mathbb Q, C\rangle)$ (discussed above) or has an element satisfying the condition (ii) of the lemma \ref{ss}, which is impossible.
\end{proof}

\textbf{$\Gamma^*$ initiates the group $\Gamma(\langle\mathbb Q, <\rangle)$ -- Order}

We show that \textbf{the group consists of all positive permutations initiating elements of  $\Gamma(\langle\mathbb Q, <\rangle)$ that preserves the order (<)}.

\begin{proof}
    
    As in previous cases every proper supergroups of this group is considered above or has a negative element then by lemma \ref{ss} (ii) it should be $Sym(\mathbb Q)$. Contradiction.  
\end{proof}

\section{Open problems}

We have already mentioned that our research on definability began with a question from P. S. Novikov, which was given to us by Albert Abramovich Muchnik. It was about describing the definability lattice for $\langle\mathbb Z ,\{+\}\rangle$. It gradually became clear that the situation is much more complicated and interesting than it seemed at first glance, based on the simplicity of the structure of the addition of integers. 

Simplifying the problem, we came to the order of rational numbers and rediscovered the result mentioned in this article on the definability lattice for this structure \cite{cam}.

The next step was to describe the lattice for $\langle\mathbb Z ,\{+1\}\rangle$\cite{we}.

This work is another advancement. At the same time, we still can not obtain a complete description of the definability lattice for $\langle\mathbb Z ,\{<\}\rangle$. The situation here also looks difficult. We see the following cases:

1. The simplest next step is to move from 1-codirection to $n$-co-direction. What does a lattice of spaces greater than some $n$-co-direction (and, of course, smaller then <) space look like?

2. Directly below the 1-co-direction in our Diagram lies the 1-neighborhood. The subgroups of the 1-neighborhood automorphism group are systemic, that is, they map each verticals onto a vertical, perhaps with a flip. What, besides the spaces described in this paper, lies between the 1-neighborhood and the order <?

3. It seems that the question of the elements of our lattice that are not comparable to the 1-neighborhood is more complicated. We do not know of any such element.

\end{document}